\documentclass{article}
\usepackage[pdftex]{hyperref}
\hypersetup{
  colorlinks   = true, 
  urlcolor     = gray, 
  linkcolor    = red, 
  citecolor   =  blue
}
\usepackage{amsmath,amsfonts,amsthm,amssymb,graphicx,xcolor}
\usepackage[all,2cell,ps]{xy}

\theoremstyle{plain}
\newtheorem{theorem}[equation]{Theorem}

\newtheorem{lemma}[equation]{Lemma}

\newtheorem*{theoremA}{Theorem A}

\newtheorem*{theorem1}{Theorem 1}

\theoremstyle{definition}

\newtheorem{remark}[equation]{Remark}

\newtheorem{question}[equation]{Question}

\newcommand{\IR}{\mathbb{R}}

\title{Aspherical $4$-Manifolds, Complex Structures, and Einstein Metrics}
\author{\small{Michael Albanese} \\ \scriptsize{University of Waterloo}\\ \footnotesize{\textsf{m3albane@uwaterloo.ca}}\and \small{Luca F. Di Cerbo}\\ \scriptsize{University of Florida}\\ \footnotesize{\textsf{ldicerbo@ufl.edu}}}
\date{}

\begin{document}

\maketitle

\begin{abstract}
We show that extended graph $4$-manifolds with positive Euler characteristic cannot support a complex structure. This result stems from a new proof of the fact that there is no compact complex surface which is homotopy equivalent to a closed real-hyperbolic $4$-manifold. Finally, we construct infinitely many extended graph 4-manifolds with positive Euler characteristic which support almost complex structures.
\end{abstract}

\vspace{12cm}

\tableofcontents

\vspace{1cm}


\section{Introduction}

There are many problems in geometry and topology for which dimension four is an extremely difficult case. One exception to this phenomenon is determining the existence (or non-existence) of a complex structure. Aside from the necessary condition of admitting an almost complex structure, our knowledge of compact complex surfaces via the Kodaira-Enriques classification is a powerful tool in many cases. An example is the following classical non-existence theorem for complex structures on $4$-manifolds with constant negative sectional curvature.

\begin{theorem}[Kotschick]\label{Kot}
There is no compact complex surface homotopy equivalent to a closed real-hyperbolic $4$-manifold.
\end{theorem}

A weaker version of this result, namely that a closed real-hyperbolic $4$-manifold cannot admit a complex structure, first appeared in a paper of Wall \cite[Theorem 10.2]{Wal85}. It was later clarified and strengthened by Kotschick to give the above statement \cite[Proposition 2]{Kot92}. The proof first uses the Kodaira-Enriques classification to reduce to the case of general type surfaces. This last possibility is then ruled out by combining results from the theory of harmonic maps due to Eells and Sampson \cite{ES64} and Sampson \cite{Sam86}. The former implies that any map from a compact K\"ahler manifold to a closed real-hyperbolic manifold is homotopic to a harmonic map, while the latter shows that any such harmonic map cannot be a homotopy equivalence in real dimension greater than two. Again using harmonic map techniques, Carlson and Toledo \cite[Corollary 1.2]{CT97} were later able to show that the fundamental group of a closed real-hyperbolic $n$-manifold, with $n\geq 3$, cannot be the fundamental group of any compact complex surface. 

After collecting some generalities about aspherical complex surfaces in Section \ref{2}, we give an alternative proof of the above theorem in Section \ref{3}. This new approach replaces the harmonic map techniques with the the Aubin-Yau result on the existence of K\"ahler-Einstein metrics and a rigidity result of Besson-Gallot-Courtois for Einstein metrics in dimension four. As mentioned above, the existence of an almost complex structure is a necessary condition for the existence of a complex structure, so one may suspect that these manifolds don't admit almost complex structures either. This is not the case; there are many examples of closed real-hyperbolic $4$-manifolds which admit almost complex structures, see Remark \ref{almostreal}.

One may wonder what are the advantages of providing a new proof of Theorem \ref{Kot} based on the theory of Einstein metrics in dimension four. Generally speaking, a new point of view on a classical result can only shed more light on closely related problems. In this case, the new proof nicely generalizes to prove a non-existence result for complex structures on a much more general class of $4$-manifolds (see Section \ref{4}). This is the \emph{main} result of this paper:

\begin{theoremA}
A closed, oriented, extended graph $4$-manifold $M^4$ with at least one pure real-hyperbolic piece cannot admit a complex structure.
\end{theoremA}

The existence of at least one pure real-hyperbolic piece is equivalent to the topological condition $\chi(M) > 0$; this assumption is necessary, see Remark \ref{necessary}. Theorem A is currently out of reach for the harmonic map techniques that featured in Kotschick's original proof of Theorem \ref{Kot}. With that said, we refer to \cite{Her01} for an extension of the ideas of Carlson and Toledo to certain graph-like manifolds in dimension three. Just as in the real-hyperbolic case, many of the manifolds in Theorem A admit almost complex structures, see Remark \ref{almostgraph}.\\

\vspace{0.1 in}
\noindent\textbf{Acknowledgments}. The authors thank Benson Farb and Dieter Kotschick for very useful bibliographical suggestions and for pertinent comments on the manuscript. In particular, Dieter Kotschick provided a useful reference to a strengthening of Remark \ref{almostreal} and several remarks which improved the clarity of the paper. The authors would also like to thank Claude LeBrun for discussions during the early stages of this work regarding Theorem \ref{Kot}. The first author would also like to thank Dennis Sullivan for some clarifying remarks. 

During the preparation of this work, the first author visited the University of Florida. He thanks this institution for the hospitality and support. During the final stages of this work, the second author visited the University of Waterloo. He thanks this institution for support and for the invitation to present this research in the Geometry and Topology seminar. This work is partially supported by NSF grant DMS-2104662.\\

\section{Aspherical $4$-Manifolds and Complex Structures}\label{2}

In this section, we collect some generalities concerning aspherical $4$-manifolds with a complex structure. Throughout, we rely on the Kodaira-Enriques classification, see \cite[Table 10]{BHPV04} for instance.

\begin{lemma}\label{Albanese0}
If a closed aspherical $4$-manifold admits a complex structure, then it has to be minimal. Moreover, if it is K\"ahler, it does not contain any rational curves.
\end{lemma}

\begin{proof}
Let $\pi : \widetilde{X} \to X$ be the universal cover of an aspherical complex surface $X$. Suppose $i : \mathbb{CP}^1 \hookrightarrow X$ is the inclusion of a rational curve in $X$. Note that $i$ lifts to a map $\mathbb{CP}^1 \to \widetilde{X}$ which is nullhomotopic as $\widetilde{X}$ is contractible; it follows that $i_*[\mathbb{CP}^1] = 0$. Every compact submanifold of a compact K\"ahler manifold is non-trivial in homology, as is the exceptional divisor of a blowup, so the result follows.  
\end{proof}

\begin{lemma}\label{Albanese1}
If a closed, oriented, aspherical $4$-manifold with 
\[
\chi>\frac{3}{2}|\tau|
\]
admits a complex structure, then it has to be minimal of general type with ample canonical line bundle.
\end{lemma}
\begin{proof}
Suppose such a manifold $X$ admits a complex structure. Note that $X$ must be minimal by Lemma \ref{Albanese0}. Furthermore, since $c_1^2(X) = 2\chi(X) + 3\tau(X)$, the stated inequality implies $c_1^2(X) > 0$. 

Note that $X$ cannot be a surface of type $\mathrm{VII}$ since $c_1(X)^2 \leq 0$ for such surfaces. Since surfaces of non-K\"ahler type of Kodaira dimension $\kappa(X) = 0, 1$ are elliptic, they satisfy $c_1(X)^2 = 0$; likewise for surfaces of K\"ahler type with $\kappa(X) = 1$. As a surface of K\"ahler type with $\kappa(X) = 0$ is finitely covered by a torus or a K3 surface, they all satisfy $c_1^2(X) = 0$ (alternatively, they are deformation equivalent to an elliptic surface). While the K\"ahler type surfaces with $\kappa(X) = -\infty$ satisfy $c_1^2(X) > 0$, they are all rational or ruled which are not aspherical. That only leaves general type.

To see that the canonical bundle of $X$ must be ample, we use the Nakai-Moishezon criterion. First note that $K_X^2 = K_X\cdot K_X = c_1(K_X)^2 = c_1(X)^2 > 0$. Now suppose, by way of contradiction, there is an irreducible complex curve $C$ in $X$ with $K_X\cdot C \leq 0$. Since $K_X$ is nef, it follows that $K_X\cdot C = 0$. Since $K_X^2 > 0$, it follows from the Hodge index theorem \cite[Corollary IV.2.16]{BHPV04} that $C^2 = C\cdot C \leq 0$. Note that equality occurs if and only if $C$ is numerically trivial, which is not the case since it has non-zero intersection with an ample divisor, so $C^2 < 0$. The arithmetic genus of $C$ is given by
\[
p_{a}(C)=1+\frac{K_{X}\cdot C+C^{2}}{2}=1+\frac{C^{2}}{2}.
\]
Since $C^2 < 0$ and $p_a(C)$ is a non-negative integer, we conclude that $C^2 = -2$ and $p_a(C) = 0$. Let $\overline{C}$ be the normalization of $C$.  Recall that  in dimension one the normalization coincides with resolution of singularities, so $\overline{C}$ is smooth. Now
\[
g(\overline{C}) = p_a(C) - \sum_i\delta_{x_i}
\]
where $g(\overline{C})$ is the topological genus, and where $\delta_{x_i}$ is the local genus drop at the singular points $x_i$ of $C$. Since $g(\overline{C})$ is non-negative and $p_a(C) = 0$, we see that $\delta_{x_i} = 0$. Therefore $C$ is a smooth rational curve with $C^2 = -2$. The existence of such a curve is prohibited by Lemma \ref{Albanese0}, so we must have $K_X\cdot C > 0$ for all curves $C$ in $X$, so $K_X$ is ample. 

\end{proof}

\section{Hyperbolic $4$-Manifolds and Complex Structures Revisited}\label{3}

In this section, we present an alternative proof of the following result of Kotschick \cite[Proposition 2]{Kot92}:

\begin{theorem1}[Kotschick]
There is no compact complex surface homotopy equivalent to a closed real-hyperbolic $4$-manifold.
\end{theorem1}

After using the Kodaira-Enriques classification to reduce to the case of general type surfaces, the original proof uses results from the theory of harmonic maps to rule out this case. The proof below only differs in the argument used to deal with general type surfaces. Instead, we combine the observations from Section \ref{2} with results regarding Einstein metrics due to Aubin and Yau, and Besson-Courtois-Gallot.

\begin{proof}[Alternative proof]
By contradiction, suppose $X$ is a compact complex surface which is homotopy equivalent to a closed real-hyperbolic $4$-manifold $M$. Since the real-hyperbolic metric on $M$ is Einstein and has vanishing Weyl curvature, the Chern-Gauss-Bonnet and Thom-Hirzebruch formulas imply that
\[
\tau(M)=0, \quad \chi(M)>0.
\]
Note that $X$ is aspherical since $M$ is. Moreover, $X$ satisfies the assumptions of Lemma \ref{Albanese1} since $\tau(X) = \tau(M) = 0$ and $\chi(X) = \chi(M) > 0$. We conclude that $X$ has to be a surface of general type with ample canonical line bundle. Due to the well-known existence result for K\"ahler-Einstein metrics of Aubin \cite{Aub76} and Yau \cite{Yau78}, $X$ supports a K\"ahler-Einstein metric with negative cosmological constant. On the other hand, a celebrated rigidity result of Besson-Courtois-Gallot \cite{BCG95} (also see \cite{BCG96}), implies that if $X$ admits an Einstein metric $g$, then $X$ is diffeomorphic to $M$ and $g$ is isometric to a real-hyperbolic metric. In particular, $g$ must be locally symmetric. Now, a locally symmetric, K\"ahler-Einstein metric with negative sectional curvature must be complex hyperbolic and as such it cannot have constant sectional curvature.
\end{proof}

\begin{remark}\label{almostreal}
While no closed real-hyperbolic $4$-manifold can admit a complex structure, there are many which admit almost complex structures. In an unpublished proof, Kotschick has shown, in response to a question from Toledo, that a closed orientable real-hyperbolic manifold admits an almost complex structure if and only if the Euler characteristic is divisible by four\footnote{This was initially communicated to us privately by Kotschick and has since appeared in print, see \cite{Kot23}.}. Such manifolds exist, e.g., we can take the manifolds constructed in \cite{CM05} and \cite{Lo08} which have Euler characteristic $16$. These appear to be the orientable, closed, real-hyperbolic $4$-manifolds with the smallest Euler characteristic currently available in the literature. By taking covers, we obtain infinitely many examples.

On the other hand, it has been conjectured by LeBrun that no closed hyperbolic $4$-manifold can admit a symplectic form \cite{LeB02}. We refer to the recent work of Agol and Lin \cite{AL20} for some progress on this conjecture and the first examples of real-hyperbolic $4$-manifolds with vanishing Seiberg-Witten invariants.
\end{remark}

\section{Extended Graph $4$-Manifolds and Complex Structures}\label{4}

In this section, we prove that \textit{extended graph $4$-manifolds} with positive Euler characteristic cannot support a complex structure. Extended graph $n$-manifolds were defined by Frigerio-Lafont-Sisto in \cite[Definition 0.2]{LafontBook}. These manifolds are decomposed into finitely many pieces referred to as the vertices. Each vertex is a manifold with boundary tori, and these tori give the edges of the graph. More precisely, the various boundaries pieces are pieced together via affine diffeomorphisms in order to obtain a closed $4$-manifold. The interior of each vertex is diffeomorphic to a finite volume real-hyperbolic $4$-manifold $\Gamma\backslash\textbf{H}^{4}_{\IR}$ with toral cusps, or to the products
\[
\Lambda\backslash\textbf{H}^{3}_{\IR}\times S^1,\qquad \Lambda^\prime\backslash\textbf{H}^{2}_{\IR}\times T^2
\]
of a lower dimensional finite volume real-hyperbolic manifold with toral cusps with a standard torus of the appropriate dimension (the product pieces). Moreover, if the Euler characteristic is positive, at least one of the vertices must be pure real-hyperbolic. Indeed, an application of the Chern-Gauss-Bonnet theorem (\emph{cf}. \cite{CG85}) and the Thom-Hirzebruch formula give the following.

\begin{lemma}\emph{\cite[Lemma 8]{DiC22}}\label{Euler}
	Let $M^4$ be a closed, oriented, extended graph $4$-manifold with $k\geq 1$ pure real-hyperbolic pieces, say $\{(\Gamma_{i}\backslash\textbf{H}^{4}_{\IR}, g_{-1})\}^k_{i=1}$. We then have
	\[
	\tau(M)=0, \quad \chi(M)=\frac{3}{4\pi^2}\sum^k_{i=1}Vol_{g_{-1}}(\Gamma_i\backslash\textbf{H}^{4}_{\IR})>0.
	\]
\end{lemma}

We can now prove the main non-existence result.

\begin{theorem}\label{graph}
	A closed, oriented, extended graph $4$-manifold $M^4$ with at least one pure real-hyperbolic piece cannot admit a complex structure.
\end{theorem}
\begin{proof}
By contradiction, let us assume $M$ admits a complex structure. By Lemma \ref{Euler}, we know that
\[
\tau(M)=0, \quad \chi(M)>0.
\]
Thus, $M$ satisfies the assumptions of Lemma \ref{Albanese0} and Lemma \ref{Albanese1}. We conclude that $M$ has to be a surface of general type with ample canonical line bundle. Because of the celebrated existence result for K\"ahler-Einstein metrics of Aubin \cite{Aub76} and Yau \cite{Yau78}, $M$ supports an Einstein metric with negative cosmological constant. On the other hand, because of \cite[Theorem 1]{DiC22} we know that closed, extended graph $4$-manifolds cannot support any Einstein metric. This contradiction implies that $M$ cannot be complex.
\end{proof}

We close with three remarks concerning the sharpness of Theorem \ref{graph}.

\begin{remark}
The proof of Theorem \ref{graph} is very similar in spirit to the proof given in Section \ref{3}. With that said, there is a crucial difference, namely here one cannot extend to the case of closed manifolds homotopy equivalent to extended graph $4$-manifolds (with at least one pure-hyperbolic piece) without further work. This is due to the fact that the non-existence result for Einstein metrics \cite[Theorem 1]{DiC22} holds true for $4$-manifolds \emph{diffeomorphic} to extended graph $4$-manifolds.
\end{remark}

\begin{remark}\label{necessary}
Note that Theorem \ref{graph} is no longer true if there are no pure real-hyperbolic pieces. In this case the Euler characteristic is zero, and it is easy to construct extended graph $4$-manifolds with $\chi=0$ supporting a complex structure. For example, any product of a higher genus closed surface with the two dimensional torus can be thought of as an extended graph $4$-manifold and it clearly supports a product complex structure.
\end{remark}

\begin{remark}\label{almostgraph}
Just as in Remark \ref{almostreal}, we can construct many extended graph 4-manifolds with positive Euler characteristic which support almost complex structures. The construction goes as follows.

First recall that a closed oriented 4-manifold $X$ admits an almost complex structure $J$, consistent with the orientation, if and only if there is a class $c \in H^2(X; \mathbb{Z})$ such that $c\equiv w_2(X) \bmod 2$ and $\langle c^2, [X]\rangle = 2\chi(X) + 3\tau(X)$. Suppose now that $X$ is spin (i.e., $w_2(X) = 0$), with $b_2(X) > 0$, and $\tau(X) = 0$. Since $X$ is spin, $c$ satisfies $c\equiv w_2(X) \bmod 2$ if and only if $c = 2\gamma$ for some $\gamma \in H^2(X; \mathbb{Z})$. So $\langle c^2, [X]\rangle = \langle (2\gamma)^2, [X]\rangle = 4\langle\gamma^2, [X]\rangle \in 8\mathbb{Z}$ as the intersection form of $X$ is even. Therefore, if $\langle c^2, [X]\rangle = 2\chi(X) + 3\tau(X) = 2\chi(X)$, we must have $4 \mid \chi(X)$. Conversely, the intersection form of $X$ (the bilinear form on $H^2(X; \mathbb{Z})/\text{torsion}$ which is induced by cup product), is isomorphic to
\[
\begin{bmatrix}
0 & 1 & & & & &\\
1 & 0 & & & & &\\
 & & 0 & 1 & & &\\
 & & 1 & 0 & & &\\
 & & & & \ddots & & \\
 & & & & & 0 & 1\\
 & & & & & 1 & 0
\end{bmatrix}
\]
by \cite[Theorem II.5.3]{MH73}. So for any even integer $m$, there is $\gamma \in H^2(X; \mathbb{Z})$ with $\langle\gamma^2, [X]\rangle = m$. Taking $m = \frac{1}{2}\chi(X)$, we see that $\langle c^2, [X]\rangle = 2\chi(X)$. That is, if $X$ is spin, $b_2(X) > 0$, and $\tau(X) = 0$, then $X$ admits an almost complex structure if and only $4\mid \chi(X)$.

Now select a finite volume real-hyperbolic manifold $M$ with Euler characteristic four -- such examples are constructed in the paper of Ratcliffe and Tschantz \cite{RT00}. By Sullivan's result \cite{Sul79} (see also \cite[Theorem 4.1]{LR20}), there exists a finite regular cover $M'$ which is spin and with $4 \mid \chi(M')$. Let $X$ be the spin double of $M'$ obtained by chopping the cusps of $M'$ and gluing with another copy equipped with the reverse orientation. Then $X$ is spin with $4\mid \chi(X)$. By Lemma \ref{Euler}, we also have $\chi(X) \neq 0$ and $\tau(X) = 0$. As in Remark \ref{almostreal}, we see that $X$ admits almost complex structures, but it cannot support a complex structure by Theorem \ref{graph}.

We believe that the proof of Kotschick mentioned in Remark \ref{almostreal} extends to the graph manifold case to give the analogous statement\footnote{As was mentioned earlier, the proof has since been published \cite{Kot23}, and indeed it extends to the graph manifold case.}. That said, we do not feel comfortable trying to reproduce or interpret his private communication to us, and we decide to leave our original (possibly weaker) result unchanged. Indeed, the main point of this remark is simply to show that almost complex extended graph $4$-manifolds with positive Euler characteristic exist in profusion.

\end{remark}
\vspace{0.2in}


\begin{thebibliography}{ELMNPM}

\bibitem[AL20]{AL20} I. Agol, F. Lin. Hyperbolic four-manifolds with vanishing Seiberg-Witten invariants. \textit{Characters in Low-Dimensional Topology}, 1--8, Contemp. Math., 760, Centre Rech. Math. Proc., \textit{Amer. Math. Soc., [Providence], RI}, 2020.

\bibitem[Aub76]{Aub76} T. Aubin. Equations du type Monge-Amp\`ere sur les vari\'et\'es k\"ahl\'eriennes compactes. \textit{C. R. Acad. Sci. Paris 283A} (1995), 119-121.

\bibitem[BHPV04]{BHPV04} W. P. Barth, K. Hulek, C. A. M. Peters, A. Van de Ven. Compact complex surfaces, 2nd ed., \textit{Ergebnisse der Mathematik und ihrer Grenzgebiete. 3. Folge}. A Series of Modern Surveys in Mathematics [Results in Mathematics and Related Areas. 3rd Series. A Series of Modern Surveys in Mathematics], vol. 4, Springer-Verlag, Berlin, 2004. 

\bibitem[BCG95]{BCG95} G. Besson, G. Courtois, S. Gallot. Entropies and rigidit\'es des espaces localment sym\'etriques de curbure strictment n\'egative. \textit{Geom. Func. Anal. 5} (1995), 731-799.

\bibitem[BCG96]{BCG96} G. Besson, G. Courtois, S. Gallot.  Minimal entropy and Mostow's rigidity theorems. \textit{Ergodic Theory Dynam. Systems 16} (1996), no. 4,  623–649.



\bibitem[CG85]{CG85} J. Cheeger, M. Gromov. Bounds on the von Neumann dimension of $L^2$-cohomology and the Gauss-Bonnet theorem for open manifolds. \textit{J. Differential Geom. 21} (1985), no. 1, 1-34.

\bibitem[CM05]{CM05} M. Conder, C. Maclachlan. Compact hyperbolic 4-manifolds of small volume. \textit{Proc. Amer. Math. Soc.} 133 (2004), no. 8, 2469--2476.

\bibitem[CT97]{CT97} J. A. Carlson, D. Toledo. On fundamental group of class VII surfaces. \textit{Bull. London Math. Soc. 29} (1997), no. 1, 98-102.


\bibitem[DiC24]{DiC22} L. F. Di Cerbo. Extended Graph 4-Manifolds, and Einstein Metrics. \textit{Ann. Math. Qu\'e. 48} (2024), no. 1, 269-276.

\bibitem[ES64]{ES64} J. Eells, J. H. Sampson. Harmonic mappings of Riemannian manifolds. \textit{Amer. J. Math. 86} (1964), 109-160.


\bibitem[FLS15]{LafontBook} R. Frigerio, J.-F. Lafont, A. Sisto. Rigidity of High Dimensional Graph Manifolds. \textit{Ast\'erisque Volume 372}, {2015}.

\bibitem[Her01]{Her01} L. Hern\'andez-Lamoneda. Non-positively curved $3$-manifolds with non-K\"ahler $\pi_1$. \textit{C. R. Acad. Sci. Paris Sér. I Math. 332} (2001), no. 3, 249–252.  



\bibitem[Kot92]{Kot92} D. Kotschick. Remarks on geometric structures on complex surfaces. \textit{Topology 31} (1992), no. 2, 317-321.

\bibitem[Kot23]{Kot23} D. Kotschick. Almost complex structures on hyperbolic manifolds. \textit{Math. Z. 305}, no. 9 (2023). 

\bibitem[LeB02]{LeB02} C. LeBrun. Hyperbolic manifolds, harmonic forms, and Seiberg-Witten invariants. \textit{Geom. Dedicata} 91 (2002), 137--154.

\bibitem[Lo08]{Lo08} C. Long. Small volume closed hyperbolic 4‐manifolds. \textit{Bull. Lond. Math. Soc.} 40 (2008), no. 5, 913--916.

\bibitem[LR20]{LR20} D. Long, A. Reid. Virtually spinning hyperbolic manifolds. \textit{Proc. Edinb. Math. Soc. (2)} 63 (2020), no. 2, 305--313.

\bibitem[MH73]{MH73} J. Milnor, D. Husemoller. Symmetric bilinear forms. Ergebnisse der Mathematik und ihrer Grenzgebiete, Band 73. \textit{Springer-Verlag, New York-Heidelberg}, 1973. viii+147 pp.


\bibitem[RT00]{RT00} J. Ratcliffe, S. Tschantz. The volume spectrum of hyperbolic 4-manifolds. \textit{Experiment. Math.} 9 (2000), no. 1, 101--125.

\bibitem[Sam86]{Sam86} J. H. Sampson. Applications of harmonic maps to Kähler geometry. \textit{Complex differential geometry and nonlinear differential equations (Brunswick, Maine, 1984)}, 125–134, Contemp. Math., 49, Amer. Math. Soc., Providence, RI, 1986. 




\bibitem[Sul79]{Sul79} D. Sullivan. Hyperbolic geometry and homeomorphisms. Geometric topology (\textit{Proc. Georgia Topology Conf., Anthens, Ga., 1977}), pp. 543--555, \textit{Academic Press, New York-London}, 1979.

\bibitem[Wal85]{Wal85} C. T. C. Wall. Geometric structures on compact complex analytic surfaces. \textit{Topology 25} (1986), no. 2, 119-153.

\bibitem[Yau78]{Yau78} S.-T. Yau. On the Ricci curvature of a compact K\"ahler manifold and the complex Monge-Amp\`ere equation. \textit{Comm. Pure Appl. Math. 31} (1978), no. 3, 339-411.

\end{thebibliography}
\end{document}